\documentclass[11pt,reqno]{amsart}
\usepackage{graphicx}
\usepackage[caption = false]{subfig}
\usepackage{amsmath,amsfonts,amssymb,comment}
\usepackage{mathrsfs}
\usepackage{siunitx}
\usepackage{float}
\usepackage{enumerate}
\usepackage{enumitem}
\usepackage{mathtools}
\usepackage{amsthm}
\usepackage{varioref}
\usepackage{float}
\usepackage{wrapfig}
\usepackage[english]{babel}
\usepackage[utf8x]{inputenc}\usepackage[english]{babel}
\usepackage{soul}
\usepackage{enumitem}
\usepackage{xcolor}
\usepackage{hyperref}
\usepackage{cite}
\usepackage{relsize} 
\usepackage{titlesec}
\usepackage{cleveref}
\usepackage{varioref}
\usepackage{aliascnt}
\usepackage{cleveref}   

\newtheorem{theorem}{Theorem}[section]

\newaliascnt{lemma}{theorem}
\newtheorem{lemma}[lemma]{Lemma}
\aliascntresetthe{lemma}

\newaliascnt{proposition}{theorem}

\aliascntresetthe{proposition}

\newaliascnt{corollary}{theorem}

\aliascntresetthe{corollary}

\theoremstyle{definition}
\newaliascnt{definition}{theorem}

\aliascntresetthe{definition}
\theoremstyle{remark}
\newaliascnt{remark}{theorem}
\newtheorem{remark}[remark]{Remark}
\aliascntresetthe{remark}

\newaliascnt{example}{theorem}

\aliascntresetthe{example}
\usepackage[margin=1in]{geometry}
\titleformat{\section}[block]
  {\normalfont\Large\bfseries}{\thesection}{1em}{}
\titlespacing*{\section}
  {0pt}{10pt plus 4pt minus 2pt}{7pt plus 4pt minus 2pt}
\titleformat{\subsection}[block]
  {\normalfont\large\bfseries}{\thesubsection}{1em}{}
\titlespacing*{\subsection}
  {0pt}{8pt plus 3pt minus 2pt}{5pt plus 2pt minus 1pt}

\theoremstyle{definition}

\theoremstyle{remark}

\numberwithin{equation}{section}

\Crefname{theorem}{theorem}{theorems}
\Crefname{lemma}{lemma}{lemmas}
\Crefname{proposition}{proposition}{propositions}
\Crefname{corollary}{corollary}{corollaries}
\Crefname{definition}{definition}{definitions}
\Crefname{remark}{remark}{remarks}
\Crefname{example}{example}{examples}

\Crefname{theorem}{Theorem}{Theorems}
\Crefname{lemma}{Lemma}{Lemmas}
\Crefname{proposition}{Proposition}{Propositions}
\Crefname{corollary}{Corollary}{Corollaries}
\Crefname{definition}{Definition}{Definitions}
\Crefname{remark}{Remark}{Remarks}
\Crefname{example}{Example}{Examples}
\begin{document}
	\title[On Geometric properties and Coefficient bounds for $\mathcal{S}^*_{B}$]{On Geometric properties and Coefficient bounds for $\mathcal{S}^*_{B}$}
	    \author[S. Sivaprasad Kumar]{S. Sivaprasad Kumar}
	\address{Department of Applied Mathematics, Delhi Technological University, Bawana Road, Delhi-110042, INDIA}
	\email{spkumar@dce.ac.in}
        \author[Arya Tripathi]{Arya Tripathi}
	\address{Department of Applied Mathematics, Delhi Technological University, Bawana Road, Delhi-110042, INDIA}
	\email{ms.aryatripathi\_25phdam03@dtu.ac.in}
\begin{abstract}
\begin{sloppypar}
This paper deals with the geometric properties of functions belonging to the class $\mathcal{S}^*_{B}$ of starlike functions associated with a balloon-shaped domain, given by   \[
        \mathcal{S}^{\ast}_{B}= \left\{ f \in \mathcal{A} : \frac{z f'(z)}{f(z)} \prec \frac{1}{1-\log (1+z)} :=B(z), \quad z \in \mathbb{D} \right\},
    \] and also derive sharp bounds for the Zalcman functionals, Krushkal inequality, third-order Hankel, Toeplitz and Hermitian-Toeplitz determinant. The sharpness of these results are verified by constructing suitable extremal functions.
    \end{sloppypar}
\end{abstract}

\maketitle

\section{\hspace{5pt} Introduction}

Let $\mathcal{A}$ denote the family of analytic functions $f$ in $\mathbb{D}=\{z\in\mathbb{C}: |z|<1\}$, normalized by these conditions $f(0)=0$ and $f'(0)=1$. For $f\in \mathcal{A}$:

    \begin{equation}
        f(z) = z+a_2z^2+a_3z^3+\dots =z + \sum_{n=2}^{\infty} a_n z^n, \label{eq:1.1}
    \end{equation}

and $\mathcal{S}\subset \mathcal{A}$ represent the class of univalent functions. A significant part of the literature has been devoted to the study of coefficient estimates and related functionals for various subclasses of $\mathcal{S}$ defined by geometric or analytic conditions. Among these, the class of starlike functions plays a fundamental role in geometric function theory. A function $f \in \mathcal{S}$ is called starlike if it maps the open unit disk $\mathbb{D}$ onto a starlike domain with respect to the origin. This class $\mathcal{S}^*$ admits the following analytic characterization:

    \[
        \mathcal{S}^*=\left\{ f\in\mathcal{S} :\Re\!\left(\frac{z f'(z)}{f(z)}\right)>0,\quad z\in\mathbb{D} \right\}.
    \]

For analytic functions $f$ and $g$, we say that $f$ is subordinate to $g$, written $f \prec g$, if $f(z)=g(w(z))$ for some Schwarz function $w$ with $w(0)=0$ and $|w(z)|<1$. Ma and Minda \cite{MaMinda1994} developed a framework for the study of starlike functions based on the concept of subordination in 1992. Let $\varphi$ be an analytic and univalent function in the unit disk $\mathbb{D}$ satisfying $\Re(\varphi(z))>0$, normalized by $\varphi(0)=1$ and $\varphi'(0)>0$, and such that $\varphi(\mathbb{D})$ is starlike with respect to the point $1$ and symmetric with respect to the real axis. The Ma-Minda class of starlike functions associated with
$\varphi$, denoted by $\mathcal{S}^*(\varphi)$ and defined as:

    \[
        \mathcal{S}^*(\varphi)=\left\{ f\in\mathcal{A} :\frac{z f'(z)}{f(z)} \prec \varphi(z), \quad z\in\mathbb{D} \right\}.
    \]
Different choices of $\varphi$ yield well-known subclasses of $\mathcal{S}^{\ast}$, see~\cite{cho_kumar_kumar_ravichandran2019, Sokol1996}. Motivated by the Ma-Minda class of starlike functions, we recall the subclass of analytic functions associated with the function 
$\varphi(z)=1/(1-\log (1+z)):=B(z)$, which was recently introduced in \cite{kumar_arya_balloon}. The class is defined as follows:
    \[
        \mathcal{S}^{\ast}_{B}= \left\{ f \in \mathcal{A} : \frac{z f'(z)}{f(z)} \prec \dfrac{1}{1-\log (1+z)}:= B(z), \quad z \in \mathbb{D} \right\},
    \] 

which is associated to a balloon-shaped domain $B(\mathbb{D})$, where $B(\mathbb{D})$ is given by:  
    \[
        B(\mathbb{D}) = \left\{ w \in \mathbb{C}\setminus \{0\} : \big|\exp\!\left(1 - \tfrac{1}{w}\right) - 1 \big| < 1 \right\}.
    \] 
The Bieberbach conjecture, as documented in \cite{Goodman1983}, has been instrumental in the evolution of univalent function theory. Motivated by this line of research, Pommerenke \cite{Pommerenke1967} introduced the concept of the $q$-th Hankel determinant in 1966. For a function $f\in\mathcal{A}$, the $q$-th Hankel determinant of order $n$ is defined by:
 
    \[
        H_{q,n}(f)=
            \begin{vmatrix}
                a_n & a_{n+1} & \cdots & a_{n+q-1} \\
                a_{n+1} & a_{n+2} & \cdots & a_{n+q} \\
                \vdots & \vdots & \ddots & \vdots \\
                a_{n+q-1} & a_{n+q} & \cdots & a_{n+2q-2}
            \end{vmatrix}, \quad q,n\in \mathbb{N}.
    \]


The second Hankel determinant $H_{2,1}(f)=a_1a_3-a_2^2$ where $a_1=1$, has been extensively studied for many subclasses of univalent functions, including starlike, convex, close-to-convex, and typically real functions. In comparison, the analysis of the third-order Hankel determinant
    
    \begin{equation}
        H_{3,1}(f) = 
            \begin{vmatrix}
                1 & a_2 & a_3 \\
                a_2 & a_3 & a_4 \\
                a_3 & a_4 & a_5
            \end{vmatrix}
            =2a_2a_3a_4+a_3a_5-a_2^2a_5-a_3^3-a_4^2,\label{H31}
    \end{equation}
    
which depends simultaneously on the initial coefficients $a_2, a_3, a_4$ and $a_5$, is significantly more involved and has gained
attention only in recent years. Several authors have obtained sharp bounds for the third-order Hankel determinant for various subclasses of univalent functions. A brief overview of some known results is given in~\Cref{tab:third_hankel}. 

\begin{table}[ht]
\centering
\begin{tabular}{|c|c|c|}
\hline
\textbf{Class} & \textbf{Sharp bound} & \textbf{References} \\
\hline
$\mathcal{S}^*(0)$& $4/9$ & Kowalczyk \emph{et al.}~\cite{Kowalczyk_and_lecko_third_hankel} \\ \hline

$\mathcal{S}^*(1/2)$& $1/9$ & Rath \emph{et al.}~\cite{Rath_and_lecko_third_hankel} \\ \hline

$\mathcal{S}^*\!(\sqrt{1+z})$ & $1/36$ & Banga \emph{et al.}\cite{Banga_third_hankel} \\ \hline

$\mathcal{S}^*(1+\arctan(z))$& $1/9$ & Kumar \emph{et al.}\cite{kumar_verma2024} \\ \hline

\end{tabular}
\vspace{0.1cm}
\caption{List of sharp bounds of the third-order Hankel determinants.}
\vspace{-0.3cm}
\label{tab:third_hankel}
\end{table}

Along with Hankel determinants, Toeplitz determinants associated with the Taylor coefficients of analytic functions have also been studied due to their relevance in coefficient problems. In this paper, we additionally obtain bounds for the third-order Toeplitz determinant for the starlike functions associated with the balloon-shaped domain. Toeplitz matrices contain same entries along their diagonals. For $f \in \mathcal{A}$, the Toeplitz determinant is given by:

    \begin{equation}
        \mathcal{T}_{q,n}(f) = 
            \begin{vmatrix}
            a_n & a_{n+1} & \cdots & a_{n+q-1} \\
            a_{n+1} & a_n & \cdots & a_{n+q-2} \\
            \vdots & \vdots & \ddots & \vdots \\
            a_{n+q-1} & a_{n+q-2} & \cdots & a_n
    \end{vmatrix},
        \quad q,n \in \mathbb{N}. \nonumber
    \end{equation}
    
Third-order Toeplitz determinant is given by:
    \begin{equation}
        \mathcal{T}_{3,1}(f) = 
            \begin{vmatrix}
                1 & a_2 & a_3 \\
                a_2 & 1 & a_2 \\
                a_3 & a_2 & 1
            \end{vmatrix}
            =1+2a_2^2a_3-2a_2^2-a_3^2.\label{toeplitz}
    \end{equation}

Several authors have obtained sharp bounds for the third-order Toeplitz determinant for various subclasses of univalent functions, see~\cite{giri2025}. Furthermore, the present investigation deals with the sharp bounds for the third-order Hermitian-Toeplitz determinant for the class $\mathcal{S}^*_{B}$. The Hermitian-Toeplitz determinant of order n is defined by:

    \begin{equation}
        H^T_{q,n}(f) = 
            \begin{vmatrix}
                a_n & a_{n+1} & \cdots & a_{n+q-1} \\
                \overline{a}_{n+1} & a_n & \cdots & a_{n+q-2} \\
                \vdots & \vdots & \ddots & \vdots \\
                \overline{a}_{n+q-1} & \overline{a}_{n+q-2} & \cdots & a_n
            \end{vmatrix},
            \quad q,n \in \mathbb{N},\nonumber
    \end{equation}

where $\overline{a}_{k}:=\overline{a_{k}}$. Third-order Hermitian-Toeplitz determinant is given by:

    \begin{equation}
        H^T_{3,1}(f) = 
            \begin{vmatrix}
                1 & a_2 & a_3 \\
                \overline{a_2} & 1 & a_2 \\
                \overline{a_3} & \overline{a_2} & 1
            \end{vmatrix}
        =2Re(a_2^2\,\overline{a_3})-2|a_2|^2-|a_3|^2+1. \label{T31f}
    \end{equation}
Upper and lower estimates for the Hermitian-Toeplitz determinant have been derived by numerous authors for a variety of subclasses of analytic and univalent functions. A summary of representative results is presented in~\Cref{tab:third_hermitian}.

\begin{table}[ht]
\centering
\begin{tabular}{|c|c|c|c|}
\hline
\textbf{Class} & \textbf{Bounds} & \textbf{References} \\
\hline
$\mathcal{S}$ & $-1 \leq|H^T_{3,1}(f)|\leq8$ &  Obradovi\'c \emph{et al.}~\cite{Obradovic_hermitian} \\ \hline
$\mathcal{S}^*(e^z)$ &$-1/15\leq|H^T_{3,1}(f)|\leq1$ & Sarkar~\cite{sarkar_hermitian_toeplitz} \\ \hline
$\mathcal{S}^*_{SG}$ & $35/64\leq|H^T_{3,1}(f)|\leq1$& Cho \emph{et al.}~\cite{cho_kumar_kumar2022} \\ \hline
$\mathcal{ST}(i)$ &$-9\leq|H^T_{3,1}(f)|\leq8$ & Jastrzebski \emph{et al.}~\cite{Jastrzebski_hermitian} \\ \hline
\end{tabular}
\vspace{0.1cm}
\caption{List of sharp bounds of the third-order Hermitian-Toeplitz determinants.}
\vspace{-0.3cm}
\label{tab:third_hermitian}
\end{table}

Coefficient problems constitute a central theme in geometric function theory, as they provide quantitative information about the analytic and geometric behavior of functions through their Taylor coefficients. In particular, determinants formed from these coefficients, such as Hankel, Toeplitz, and Hermitian-Toeplitz determinants, measure nonlinear dependencies among coefficients and play a crucial role in distinguishing various subclasses of univalent functions. Motivated by these developments, the present paper is devoted to a detailed study of the starlike class associated with the balloon-shaped domain $\mathcal{S}^*_{B}$. We begin by investigating geometric properties of the associated function $B(z)$ for $z \in \mathbb{D}$, including bounds, radius of convexity, and disk containment result describing both the largest disk contained in $B(\mathbb{D})$ and the smallest disk containing $B(\mathbb{D})$. Subsequently, we obtain bounds for the Zalcman functional $Z_n(f)$ for $n=2$ and $n=3$, as well as estimates related to the Krushkal inequality $K_{n,m}(f)$ for the cases $n=4$ and $n=5$. We further establish sharp bounds for the third-order Hankel determinant $|H_{3,1}(f)|$, third-order Toeplitz determinant $|T_{3,1}(f)|$, and third-order Hermitian-Toeplitz determinant $|H^T_{3,1}(f)|$. All results obtained are shown to be sharp, with extremal functions identified explicitly. The methods developed here are expected to be applicable to related subclasses defined through other geometric or analytic conditions.

\section{\hspace{5pt} Preliminary results}\label{3}
Let $\mathcal{P}$ be the Carath\'eodory class of functions with positive real part in $\mathbb{D}$. The Taylor series expansion of $f \in \mathcal{P}$ is defined as:
    \begin{equation}
        p(z) =1+p_1z+p_2z^2+\dots= 1 + \sum_{n=1}^{\infty}p_n z^n.\label{pp}
    \end{equation}
The Carath\'eodory class and its associated coefficient bounds play a crucial role in establishing the sharp bounds for the Hankel determinant. This section provides key lemmas that form the foundation for the main results.

\begin{lemma}\label{lem1}\cite{caratheodory1907,Pommerenke1975}
Let \( p \in \mathcal{P} \)and of the form \eqref{pp}. Then, the following inequality hold true
\begin{equation}
|p_{n+k} - \mu p_n p_k| \leq 2 \max \{ 1, |2\mu - 1| \} =
\left\{
\begin{array}{ll}
2, & \text{if } 0 \leq \mu \leq 1, \\[6pt]
2|2\mu - 1|, & \text{Otherwise}.
\end{array}
\right.\nonumber
\end{equation}
\end{lemma}

\begin{lemma} \label{lemma2.1} \cite{kwon_adam_p1p2lemma, libera_p1p2lemma}
Let \( p \in \mathcal{P} \). Then
    \begin{eqnarray}
        p_2 &=& \frac{1}{2}\left(p_1^2 + \gamma (4 - p_1^2)\right),\nonumber\\
        p_3 &=& \frac{1}{4}\Big(p_1^3 + 2p_1(4 - p_1^2) \gamma - p_1(4 - p_1^2) \gamma^2 + 2(4 - p_1^2)(1 - |\gamma|^2)\eta\Big), \nonumber\\
        p_4 &=& \frac{1}{8}\Big(p_1^4 + (4 - p_1^2) \gamma \bigl(p_1^2(\gamma^2 - 3\gamma + 3) + 4\gamma\bigr)- \nonumber\\
        & & \quad  4(4 - p_1^2)(1 - |\gamma|^2)\bigl(p_1(\gamma - 1)\eta + \overline{\gamma}\eta^2 - (1 -|\eta|^2)\rho\bigr)\Big) \nonumber,
    \end{eqnarray}
for some \( \rho, \gamma \) and \( \eta \) such that \( |\rho| \leq 1, \, |\gamma| \leq 1 \) and \( |\eta| \leq 1 \).
\end{lemma}

\begin{lemma}\label{lem3}\cite{zaprawa2021,zaprawa2021thirdhankel}
Let $w \in \mathcal{H}$, are said to be schwarz function such that $w(0)=0$ and $|w(z)|<1$ for all $z \in \mathbb{D}$, and have the following series:
\begin{equation}
w(z)=\sum_{n=1}^{\infty}b_nz^n \label{schwarzfunc}
\end{equation}
Then, the following inequalities hold true
\begin{eqnarray}
|b_1|  &\leq & 1,\nonumber\\
|b_2| & \leq & 1-|b_1|^2,\nonumber \\
|b_3| &\leq &  1-|b_1|^2-\frac{|b_2|^2}{1+|b_1|},\nonumber\\
|b_4|&\leq&1-|b_1|^2-|b_2|^2.\nonumber
\end{eqnarray}
\end{lemma}

\section*{Geometric properties of \texorpdfstring{$\boldsymbol{B(z)=1/(1-\log(1+z))}$}{1/(1-\log(1+z)}}
The function $B(z)=1/(1-\log(1+z)$ is starlike in $\mathbb{D}$, since 
$\Re(z B'(z)/B(z))>0$ for $|z|\in \mathbb{D}$, and it is symmetric about the real axis as 
$B(\overline{z})=\overline{B(z)}$. Our first result aims in finding the
radius of convexity of the function $B(z)$.

\begin{theorem}
    The radius of convexity for the function $B(z)$ is given by \[r_c=1-1/e \approx 0.63212.\]
\end{theorem}

\begin{proof}
    To find the radius of convexity of the function $B(z)$, we have to find $r_c$ such that
    \begin{equation}
        \Re \left(1+\dfrac{zB''(z)}{B'(z)}\right)>0, \; |z|<r_c. \label{convexradius}
    \end{equation}
    Since 
    \begin{equation}
        1+\dfrac{zB''(z)}{B'(z)}=\dfrac{2z+1-\log(1+z)}{(1+z)(1-\log(1+z))}. \nonumber
    \end{equation}
    Taking $z=re^{\iota \theta}$, then we get
    \begin{equation}
        \Re \left(1+\dfrac{zB''(z)}{B'(z)}\right)=1+r\dfrac{\left(\cos \theta\left(1-\tfrac{1}{4}(\log(X))^2\right)+\left(\arctan(Y)\right)^2-\log(X) \arctan(Y) \sin\theta\right)}{X\left(\left(1-\tfrac{1}{2}\log(X)\right)^2+(\arctan(Y))^2\right)}, \nonumber
    \end{equation}
    where
    \begin{equation}
        X=1+2r\cos \theta+r^2, \quad Y=\dfrac{r\sin \theta}{1+r \cos \theta}. \nonumber
    \end{equation}
    Now take \[\Re \left(1+\dfrac{zB''(z)}{B'(z)}\right):=g(r,\theta)=\dfrac{g_n(r, \theta)}{g_m(r,\theta)}.\]
    Since $g(r, \theta)=g(r,-\theta)$, the function $g(r,\theta)$ is symmetric about real axis. Hence, it suffices to consider $\theta\in[0,\pi]$. Further, we observe that for $r\in(0,1)$ and $\theta\in[0,\pi]$,
    \begin{equation}
        g_n(r, \theta):=(1+2r\cos \theta+r^2)\left(\left(1-\tfrac{1}{2}\log(1+2r\cos \theta+r^2)\right)^2+\left(\arctan\left(\dfrac{r\sin \theta}{1+r \cos \theta}\right)\right)^2\right)>0. \nonumber
    \end{equation}
    A mathematical computation shows that $g_n(r,\theta)$ attains its minimum at $\theta=\pi$, hence $g_n(r,\pi)=((\log(1-r))^2-1)$. Now we have to show that $g_n(r,\theta)>0$, i.e.,
    \begin{equation}
        ((\log(1-r))^2-1)>0. \nonumber
    \end{equation}
    Since $r>0$, $\log(1-r)>1$ is not possible. Hence, we get 
    \begin{equation}
        \log(1-r)>-1 \implies r<r_c=1-1/e \approx0.63212. \nonumber
    \end{equation}
    Therefore, \eqref{convexradius} holds for $r_c\approx0.63212.$ This completes the proof.
\end{proof}

\begin{lemma}\textbf{(Function Bounds)}
    Let $B_R(\theta)$ and $B_I(\theta)$ denote the real and imaginary part of the function $B(e^{\iota \theta})$ respectively and given as 
    \begin{equation}
        B_R(\theta)=\dfrac{1-\log(2\cos{\theta/2})}{(1-\log(2\cos{\theta/2}))^2+(\theta/2)^2}, \quad \text{and} \quad
        B_I(\theta)=\dfrac{\theta/2}{(1-\log(2\cos{\theta/2}))^2+(\theta/2)^2}. \label{realandimaginarypartofb(z)}
    \end{equation}
    Then we have
    \begin{enumerate}[label=(\roman*)]
        \item $0<\Re(B(z))<1/(1-\log2)\approx3.25889$ at $\theta=0$ and $\pi$, we get maximum and minimum of $B_R(\theta)$ respectively.  
        \item $|\Im(B(z))|\leq B_I(\theta_0)\approx 1.41379$ at $\theta_0=0.524085.$
        \item $|\arg(B(z))|\leq\left|\arctan\left(\tfrac{\theta_0/2}{1-\log(2\cos{\theta_0/2})}\right)\right|\approx0.88329$ at $\theta_0=1.37501.$
        \item $|B(z)|\leq \left|\dfrac{1}{1-\log(1+r)} \right|,$ where $|z|=r<1.$
    \end{enumerate}
    These bounds are best possible.
\end{lemma}

We now study the geometric extent of the image domain $B(\mathbb{D})$ via disk containment. In particular, we determine the largest disk contained in $B(\mathbb{D})$ and the smallest disk containing $B(\mathbb{D})$.

\begin{theorem}
    For $B(z)=1/(1-\log(1+z))$, then we have
    \begin{enumerate}
        \item
    $\mathbb{D}_L:=\{w:\;|w-a|<r_a\} \subset B(\mathbb{D})$, where
    
    \[r_a= \displaystyle\left\{\begin{array}{ll}
        \sqrt{d(\theta_a)}, & 0<a<\tfrac{2-\log2}{(1-\log2)(3-\log2)}, \\ \\
        \tfrac{1}{1-\log2}-a, & \tfrac{2-\log2}{(1-\log2)(3-\log2)}\leq a<\tfrac{1}{1-\log2}.
    \end{array}\right. \] 
    and $\theta_a \in (0,\pi)$, is the root of the given equation
    \begin{equation}
    \sin{B}(a(A^2-B^2)-A)+B\cos{B}(2aA-1)=0, \label{AandB}
    \end{equation}
    with $A:=1/(1-\log{(2\cos(\theta/2))}),\;B:=\theta/2.$\\

    \item $B(\mathbb{D}) \subset\{w:\;|w-a|<R_a\}=:\mathbb{D}_S$, where
    \[R_a= \displaystyle\left\{\begin{array}{ll}
        \tfrac{1}{1-\log2}-a,&  0 <a<\tfrac{1}{2(1-\log2)}, \\ \\
        a, & \tfrac{1}{2(1-\log2)}\leq a<\tfrac{1}{1-\log2}.
    \end{array}\right. \]

    \end{enumerate}

\end{theorem}
    
\begin{proof}
    Let $B(e^{\iota \theta})=B_R(\theta)+\iota B_I(\theta)$, where $B_R(\theta)$ and $B_I(\theta)$ are defined in \eqref{realandimaginarypartofb(z)}, represents the boundary of $B(\mathbb{D})$ and is symmmetric about real axis. Take $d(\theta)$ denote the square of the distance of $(a,0)$ from the points on the curve $B(e^{\iota \theta})$ is given by
    \[
    d(\theta):=a^2+\dfrac{1-2aA}{A^2+B^2},
    \]
    where $A$ and $B$ are given in \eqref{AandB}.
    \begin{enumerate}
        \item 
    We'll start with $\{w:\;|w-a|<r_a\} \subset B(\mathbb{D})$. For $0<a<\tfrac{2-\log2}{(1-\log2)(3-\log2)}$, 
        \[
        d'(\theta)=\dfrac{\sin{B}(a(A^2-B^2)-A)+B\cos{B}(2aA-1)}{\cos{B}(A^2+B^2)^2}=0,
        \]
        A mathematical computation shows that $d'(\theta)$ changes sign exactly at $\theta_a$, where $\theta_a$ is the root of the equation \[\sin{B}(a(A^2-B^2)-A)+B\cos{B}(2aA-1)=0.\]
        Since $d'(\theta)<0$ for $\theta \in (0,\theta_a)$ and $d'(\theta)>0$ for $\theta \in (\theta_a,\pi)$, we get $d(\theta)$ is decreasing in $(0,\theta_a)$ and increasing in $(\theta_a,\pi)$. For the largest disk inside $B(\mathbb{D})$, we need $r_a=\min_{\theta\in [0,\pi]}\sqrt{d(\theta)}$. Therefore
        \[
        \min_{\theta \in [0,\pi]}d(\theta)=d(\theta_a) \implies r_a=\sqrt{d(\theta_a)}.
        \]
        For $\tfrac{2-\log2}{(1-\log2)(3-\log2)}\leq a<\tfrac{1}{1-\log2}$, $d'(\theta)\ge 0$ for all $\theta \in (0,\pi)$, hence $d(\theta)$ is monotonic increasing. Therefore, we get 
        \[
        \min_{\theta\in [0,\pi]}d(\theta)=d(0)=\left(\dfrac{1}{1-\log2}-a\right)^2 \implies r_a=\dfrac{1}{1-\log2}-a.
        \]
    
        \item For $B(\mathbb{D}) \subset \{w:\;|w-a|<R_a\}$, we need to find the smallest disk that contain $B(\mathbb{D})$, we need $R_a=\sup_{\theta\in [0,\pi]}\sqrt{d(\theta)}$. For $0<a<\tfrac{1}{2(1-\log2)}$, the supremum of $d(\theta)$ over $[0,\pi]$ is 
        \[
        \sup_{\theta\in[0,\pi]}d(\theta)=d(0)=\left(\dfrac{1}{1-\log2}-a\right)^2 \implies R_a=\dfrac{1}{1-\log2}-a.
        \]
        For $\tfrac{1}{2(1-\log2)}\leq a<\tfrac{1}{1-\log2}$, $d(\theta)\leq a^2$ for all $\theta \in (0,\pi)$. Hence $R_a=a.$
        \end{enumerate}
        This completes the proof.
        
\end{proof}

\begin{remark}
            The largest disk $\mathbb{D}_L:=\{w:\;|w-a_1|<r_{a_1}\}$ contained in $B(\mathbb{D})$ is obtained when $a_1\approx1.90$ and $r_{a_1}=1/(1-\log2) - 1.90\approx 1.3589$, such that $\mathbb{D}_L \subset B(\mathbb{D})$. And the smallest disk $\mathbb{D}_S:=\{w:\;|w-a_2|<R_{a_2}\}$ containing $B(\mathbb{D})$, is obtained when $a_2\approx1.6295$ and $R_{a_2}\approx1.6295$, such that $B(\mathbb{D}) \subset\mathbb{D}_S$ as shown in Figure \ref{fig:diskcontain}. 
        \end{remark}
    \begin{figure}[H]
        \centering
    \includegraphics[width=0.5\textwidth]{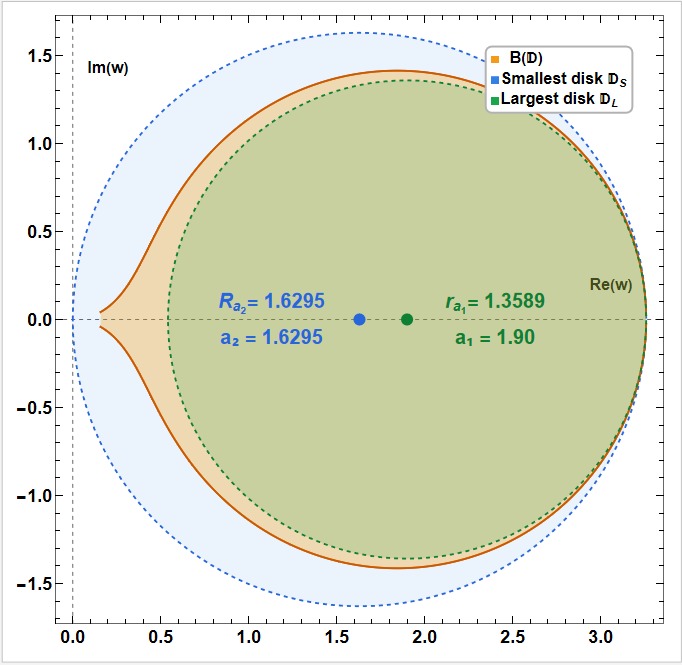}
    \caption{Visualization of the image $B(\mathbb{D})$ under the mapping 
$B(z)=1/(1-\log(1+z))$, together with the largest inscribed disk
and the smallest disk containing $B(\mathbb{D})$. The parameters $a_1$ and $a_2$ denote 
the centers of the respective disks on the real axis.}
    \label{fig:diskcontain}
    \end{figure}

\section*{Zalcman Conjecture}
Zalcman proposed a conjecture in 1960 for functions in the class $\mathcal{S}$, stating that
\[|Z_n(f)| := |a_n^2 - a_{2n-1}| \le (n - 1)^2, \quad n \ge 2,\]
with equality holds for the Koebe function $k(z) = z/(1 - z)^2$ and its rotations. 
We first estimate the Zalcman inequality for $n = 2$.
\begin{theorem}
    Let $f \in \mathcal{S}^*_{B}$, then 
        \begin{equation}
            |a_2^2-a_3|\leq \frac{1}{2}. \nonumber
        \end{equation}
    This inequality is sharp.
\end{theorem}
\begin{proof}
    Let $f \in \mathcal{S}^*_{B}$. Then we can find a Schwarz function $w(z)$ such that
        \begin{equation}
            \dfrac{zf'(z)}{f(z)}= \dfrac{1}{1-\log(1+w(z))}. \label{eq:2.3}
        \end{equation}
    Suppose $w(z)=(p(z)-1)/(p(z)+1)$, where $p \in \mathcal{P}$ given by \eqref{pp}. Then by using \eqref{eq:1.1} and \eqref{pp} in \eqref{eq:2.3}, we obtain
        \begin{eqnarray}
            a_{2}&=&\dfrac{1}{2}p_{1},\label{a2}\\
            a_{3}&=&\dfrac{1}{16}(p_1^2+4p_{2}),\label{a3}\\
            a_{4}&=&\dfrac{1}{288}\left(p_{1}^{3}+12p_{1}p_{2}+48p_{3}\right),\label{a4}\\
            a_5&=&-\dfrac{1}{4608}\left(7p_1^4-24p_1^2p_2-96p_1p_3-576p_4\right). \label{a5}
        \end{eqnarray}
        Now, using \eqref{a2} and \eqref{a3}, we get
        \begin{equation}
            a_2^2-a_3=\dfrac{1}{16}(3p_1^2-4p_2)=-\dfrac{1}{4}\left(p_2-\dfrac{3}{4}p_1^2\right). \label{a2-a3^2}
        \end{equation}
            Now using lemma \ref{lem1} in \eqref{a2-a3^2}, we get
            \begin{equation}
                |a_2^2-a_3|\leq \dfrac{1}{4}\left|p_2-\dfrac{3}{4}p_1^2\right|\leq\dfrac{1}{2}.
            \end{equation}
            This bound is sharp and equality hold for the function $f_2(z) \in \mathcal{S}^{\ast}_{B}$, defined as
            \begin{equation}
            f_{2}(z) = z \exp\left(\int_{0}^{z} \dfrac{\log(1+t^2)}{t(1-\log(1+t^2))} \, dt\right) =z + \dfrac{1}{2} z^3 + \dfrac{1}{4} z^5 + \dfrac{5}{36}z^{7}\cdots.
            \end{equation}
            This completes the proof.
\end{proof}

Now we obtain the Zalcman inequality for $n = 3$.
\begin{theorem}
    Let $f \in \mathcal{S}^*_{B}$, then 
        \begin{equation}
            |a_3^2-a_5|\leq \frac{1}{4}. \nonumber
        \end{equation}
    This inequality is sharp.
\end{theorem}
\begin{proof}
    Let $f\in \mathcal{S}^{\ast}_{B}$. Then, there exists Schwarz function $w(z)$ satisfying \eqref{schwarzfunc}. By solving and comparing the coefficients of $f(z)$ and $w(z)$  from \eqref{eq:2.3}, we obtain
\begin{equation}
    a_2=b_1,\quad a_3=\dfrac{1}{4}(3b_1^2+2b_2),\quad a_4= \dfrac{1}{36}(19b_1^3+30b_1b_2+12b_3),\label{a2a3a4 in terms of w}
\end{equation}
\begin{equation}
    a_5=\dfrac{1}{288}(101b_1^4+276b_1^2b_2+72b_2^2+168b_1b_3+72b_4). \label{a5inw}
\end{equation}
    Now using \eqref{a2a3a4 in terms of w} and \eqref{a5inw}, we get
    \begin{equation}
        a_3^2-a_5=\dfrac{1}{288}(72b_4+168b_1b_3+60b_1^2b_2-60b_1^4). \label{a3^2-a5}
    \end{equation}
    Now using lemma \ref{lem3} in \eqref{a3^2-a5}, we obtain
    \begin{eqnarray}
        |a_3^2-a_5|&\leq&\dfrac{1}{288}\left|72b_4+168b_1b_3+60b_1^2b_2-61b_1^4\right|, \nonumber\\
        &\leq&\dfrac{1}{288}\left(72|b_4|+168|b_1||b_3|+60|b_1|^2|b_2|+61|b_1|^4\right), \nonumber\\
        &\leq&\dfrac{1}{288}\Big(72(1-|b_1|^2-|b_2|^2)+168|b_1|\Big(1-|b_1|^2-\dfrac{|b_2|^2}{1+|b_1|}\Big)+60|b_1|^2|b_2|+61|b_1|^4\Big). \nonumber
    \end{eqnarray}
    Now taking $x:=|b_1|$ and $y:=|b_2|$ in above equation, we obtain
    \begin{eqnarray}
        |a_3^2-a_5|&\leq&\dfrac{1}{288}\Big(72(1-x^2-y^2)+168x\Big(1-x^2-\dfrac{y^2}{1+x}\Big)+60x^2y+61x^4\Big)\nonumber\\
        &\leq&\dfrac{1}{288}\left(168x(1-x^2)+72(1-x^2)+61x^4+60x^2y-72y^2-168x\left(\dfrac{y^2}{1+x}\right)\right)\nonumber\\
        &\leq&\dfrac{1}{288}\left(168x(1-x^2)+72(1-x^2)+61x^4+60x^2y-\left(72+\dfrac{168x}{1+x}\right)y^2\right):=G(x,y) \nonumber
        \end{eqnarray}
        Hence, we get 
        \begin{equation}
         |a_3^2-a_5|\leq G(x,y). \label{Gx,y}
    \end{equation}
    Now we need to find the maximum value $G(x,y)$ in the region $\Omega:=\{(x,y):\;x\ge0,\;y\ge0,\;y\le1-x^2\}$. First we'll check interior of $\Omega$. From $\partial G/\partial y=0$, we obtain
    \[y=\dfrac{30x^2(1+x)}{240x+72}:=y^*\]
    Since $y^*\not\le 1-x^2$ for all $x\in[0,1]$, hence there doesn't exists solution.
    Now we'll check for $y=0$ and $x=1$, we get $G(1,0)=61/288$. and For $y=0$ and $x=0$, we get $G(0,0)=72/288=1/4$. Hence from \eqref{Gx,y}, we get
    \[
    |a_3^2-a_5|\leq\dfrac{1}{4}.
    \]
    This inequality is sharp and the extremal for this inequality is obtained by the function $f_{3}(z)\in\mathcal{S}^{\ast}_{B}$, defined as
    \begin{equation}
f_{3}(z)
= z \exp\left(\int_{0}^{z} \dfrac{\log(1+t^4)}{t(1-\log(1+t^4))} \, dt\right)
= z + \dfrac{1}{4} z^5 + \dfrac{1}{32} z^9 + \cdots. \label{f2}
\end{equation}
Hence this completes the proof.
   \end{proof}

\section*{Hankel determinant}
Now we obtain sharp bound for the third-order Hankel determinant of functions associated with the class $\mathcal{S}^*_{B}$. The extremal function attaining equality is also identified.

\begin{theorem}
    Let $f \in \mathcal{S}^*_{B}$, then 
        \begin{equation}
            |H_{3,1}(f)|\leq \frac{1}{9}. \nonumber
        \end{equation}
    This inequality is sharp.
\end{theorem}

\begin{proof}
    Let $f \in \mathcal{S}^*_{B}$. Using \eqref{a2}, \eqref{a3}, \eqref{a4} and \eqref{a5} in \eqref{H31}, we get
        \begin{equation}
            H_{3,1}(f)=2a_2a_3a_4-a_2^2a_5-a_3^3-a_4^2+a_3a_5. \label{H31a}
        \end{equation}
    Now from equations \eqref{a2}, \eqref{a3}, \eqref{a4} and \eqref{a5} in \eqref{H31a}, we get
        \begin{eqnarray}
            H_{3,1}(f)=\frac{1}{663552}\Big(163p_1^6-732p_1^4p_2+3552p_1^3p_3+21888p_1p_2p_3-\nonumber\\
            \quad 576p_1^2(2p_2^2+27p_4)-1152(9p_2^3+16p_3^2-18p_2p_4)\Big). \label{h31p}
        \end{eqnarray}
    Now using ~\ref{lemma2.1} in \eqref{h31p} and taking $p:=p_1$, we get
        \begin{equation}
            H_{3,1}(f)=\frac{1}{663552}\Big(\nu_1(p,\gamma)+ \nu_2(p,\gamma)\,\eta+ \nu_3(p,\gamma)\,\eta^2+ \psi(p,\gamma,\eta)\,\rho\Big), \label{11}
        \end{equation}
    where 
        \begin{eqnarray}
            \nu_1(p,\gamma)&=&37p^6-102p^4\gamma(4-p^2)-48p^2\gamma^2(4-p^2)^2-96p^2\gamma^2(4-p^2)-\nonumber \\
            & & \quad 72p^2\gamma^3(4-p^2)^2-2592p^2\gamma^3(4-p^2)+144p^2\gamma^4(4-p^2)^2, \nonumber\\
            \nu_2(p,\gamma)&=&48p(4-p^2)(1-|\gamma|^2)(p^2 +120\gamma+24p^2\gamma-12\gamma^2(4-p^2)),\nonumber\\
            \nu_3(p,\gamma)&=&288(4-p^2)(1-|\gamma|^2)(9p^2\overline{\gamma}-16(4-p^2)+16\gamma^2(4-p^2)-18|\gamma|^2(4-p^2)),\nonumber\\
            \psi(p,\gamma,\eta)&=&2592(4-p^2)(1-|\gamma|^2)(1-|\eta|^2)(2\gamma(4-p^2)-p^2).\nonumber
        \end{eqnarray}
    Assume \(x := |\gamma|\), \(y := |\eta|\) and since \(|\rho| \le 1\), the expression \eqref{11} reduces to
        \[
            \left| H_{3,1}(f) \right|\le \frac{1}{663552}\left(|\nu_1(p,\gamma)|+ |\nu_2(p,\gamma)|\, y+ |\nu_3(p,\gamma)|\, y^2+ |\psi(p,\gamma,\eta)|\right)\le F(p,x,y),
        \]
    where
        \begin{equation}
            F(p,x,y)= \frac{1}{663552}\Big(g_1(p,x)+ g_2(p,x)\, y+ g_3(p,x)\, y^2+ g_4(p,x)\,(1 - y^2)\Big), \label{G}
        \end{equation}
    with
        \begin{eqnarray}
            g_1(p,x)&:=&37p^6+102p^4x(4-p^2)+48p^2x^2(4-p^2)^2+96p^2x^2(4-p^2)+\nonumber \\
            & & \quad 72p^2x^3(4-p^2)^2 +2592p^2x^3(4-p^2)+144p^2x^4(4-p^2)^2, \nonumber\\
            g_2(p,x)&:=&48p(4-p^2)(1-x^2)(p^2 +120x+24p^2x+12x^2(4-p^2)),\nonumber\\
            g_3(p,x)&:=&288(4-p^2)(1-x^2)(9p^2x+16(4-p^2)+34x^2(4-p^2)),\nonumber\\
            g_4(p,x)&:=&2592(4-p^2)(1-x^2)(2x(4-p^2)+p^2).\nonumber
        \end{eqnarray}
    The function $F(p,x,y)$ is maximized over the cuboid $S: = [0,2] \times [0,1] \times [0,1]$ by analyzing its interior and boundary points, including all faces and edges. 
    \begin{enumerate}
    \item \textbf{On Interior of $S$.}\\
        We first analyze the interior of $S$. For $(p,x,y) \in (0,2) \times (0,1) \times (0,1),$ the interior extrema of $F(p,x,y)$ is determined by partially differentiating \eqref{G} with respect to $y$. After simplification, we obtain
            \begin{eqnarray}
                \dfrac{\partial F}{\partial y}&=&\dfrac{1}{13824}(4-p^2)(1-x^2)\Big(p^2(1-300y+12x(2+27y)-\nonumber\\
                & & \quad 12x^2(1+34y))+24(x(5-36y)+32y+x^2(2+68y))\Big)=0, \nonumber
            \end{eqnarray}
        which yields 
            \begin{equation}
                y=\dfrac{24x(5+2x)+p^2(1+24x-12x^2)}{12(8(9x-17x^2-8)+p^2(25-27x+34x^2))}:=y_{0} \nonumber
            \end{equation}
        For the existence of a critical point in the interior of $S$, it is necessary that $y_{0} \in (0,1)$. This condition is satisfied only if the following inequalities hold simultaneously:
            \begin{equation}
                p^3(1-6x)+(4-p^2)(12px^2-30px+300-324x+408x^2)<432(1-x), \label{2.6}
            \end{equation}
        and
            \begin{equation} 
                25p^2>(4-p^2)(34x^2-27x)-80x^2+36x+64. \label{2.7}
            \end{equation}
        However, there exists no solution that satisfies both inequalities \eqref{2.6} and \eqref{2.7}. Hence, the function $F(p,x,y)$ admits no critical point in $(0,2) \times (0,1) \times (0,1)$.

    \item \textbf{Interior of the faces of $S$.}\\
        We now investigate the extrema of $F(p,x,y)$ on the interiors of the six faces of the cuboid $S$.
            \begin{enumerate}
                \item On the face $p=0, \; F(p,x,y)$ reduces to 
                    \begin{eqnarray*}
                         h_1(x,y)&:=&F(0,x,y)=\dfrac{1}{144}(1-x^2)\Big(16y^2+2x^2y(1+17y)+\\
                         & & \quad \quad \quad \quad \quad \quad x(18+5y-18y^2)\Big).
                    \end{eqnarray*}
        On solving $\partial h_1/\partial y=0$,
            \begin{equation}
                \dfrac{\partial h_1}{\partial y}=\dfrac{1}{144}(1-x^2)\Big(x(5-36y)+32y+x^2(2+68y)\Big)=0,\nonumber
            \end{equation}
        which yields 
            \begin{equation*}
                y=\dfrac{x(5+2x)}{4(9x-17x^2-8)}, \quad x\in(0,1).
            \end{equation*}
\begin{figure}[H]
        \centering
    \includegraphics[width=0.5\textwidth]{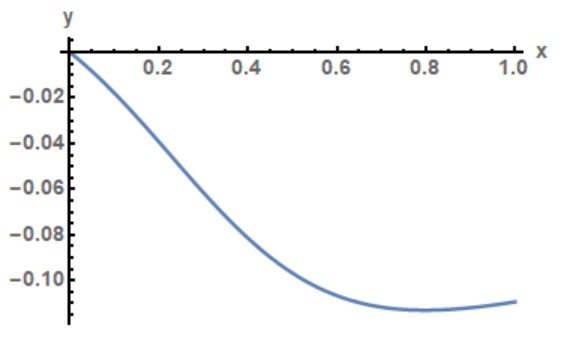}
    \caption{The graph of $y$ on $(0,1)$.}
    \label{fig:h11}
    \end{figure}

Hence, there does not exist any critical points in $(0,1)\times(0,1)$.
                \item On the face $p=2, \; F(p,x,y)$ reduces to
                \begin{equation*}
                    F(2,x,y)=\dfrac{37}{10368}\approx0.00356867. 
                \end{equation*}
                
                \item On the face $x=0, \; F(p,0,y)$ reduces to
                \begin{eqnarray}
                    h_2(x,y)&:=&F(p,0,y)=\dfrac{1}{663552}\Big(37p^6+48p^3(4-p^2)+\nonumber\\
                    & &\quad \quad \quad \quad \quad \quad 4608(4-p^2)^2y^2+2592p^2(4-p^2)(1-y^2)\Big). \nonumber
                \end{eqnarray}
                A direct computation shows that the equations $\partial h_2/\partial y = 0$ and $\partial h_2/\partial p = 0$, admit no solution  in $(0,2)\times(0,1)$.
    
                \item On the face $x=1, \; F(p,x,y)$ reduces to
                \begin{equation}
                    h_3(p):=F(p,1,y)=\dfrac{1}{663552}\Big(p^2(14976-4392p^2+199p^4)\Big).\nonumber
                \end{equation}
                To determine its extreme values, we differentiate $h_3$ with respect to $p$ and consider the equation $\partial h_3/\partial p=0$, we get
                \begin{equation*}
                    \dfrac{\partial h_3}{\partial p}=\dfrac{1}{110592}\Big(p(4992-2928p^2+199p^4)\Big)=0 \implies p\approx1.4029325\in(0,2).
                \end{equation*}
    \begin{figure}[H]
        \centering
    \includegraphics[width=0.5\textwidth]{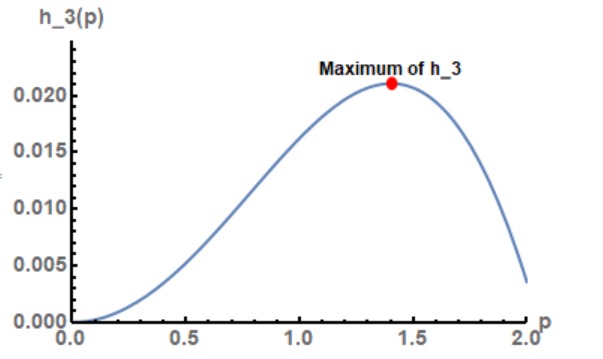}
    \caption{The graph of $h_3$ with maximum at $p_0$.}
    \label{fig:h}
    \end{figure}

                \item On the face $y=0, \; F(p,x,y)$ reduces to
                \begin{eqnarray*}
                    h_4(p,x):&=&F(p,x,0)\\
                    &=&\dfrac{1}{663552}\Big(82944\,x(1 - x^2) +1152\,p^2(9 - 36x - 8x^2 + \\
                    & &46x^3 + 2x^4)-24\,p^4(108 - 233x - 88x^2 +348x^3+\\
                    & &  48x^4)+p^6(37 - 102x + 48x^2 +  72x^3 + 144x^4)\Big).
                \end{eqnarray*}
                A computation shows that 
                \begin{eqnarray*}
                    \dfrac{\partial h_4}{\partial p}&=&\dfrac{1}{110592}p\Big( 384\,(9 - 36x - 8x^2 + 46x^3 + 2x^4)-16\,p^2(108 - \\ & &233x - 88x^2 + 348x^3 + 48x^4)+p^4(37 - 102x + 48x^2 +\\
                    & & 72x^3 + 144x^4)\Big),
                \end{eqnarray*}
                and 
                \begin{eqnarray*}
                    \dfrac{\partial h_4}{\partial x}&=&\dfrac{1}{663552}\Big( 82944(1 - x^2) -165888\,x^2-1152\,p^2(36 + 16x - \\
                    & &138x^2 - 8x^3)+ 24\,p^4(233 + 176x - 1044x^2 - 192x^3)-\\
                    & &p^6(102 - 96x - 216x^2 - 576x^3)\Big).
                \end{eqnarray*}
                A numerical analysis indicates that the system of equations $\partial h_4/\partial x = 0$ and $\partial h_4/\partial p = 0$ admit no solution in $(0,2) \times (0,1)$.

                \item On the face $y=1, \; F(p,x,y)$ reduces to
                \begin{eqnarray*}
                    h_5(p,x):=F(p,x,1)&=&\dfrac{1}{663552}\Big(4608\,p\,x(5 + 2x - 5x^2 - 2x^3)-\\
                     & & 48\,p^5(1 +24x - 13x^2 - 24x^3 +12x^4)+ \\
                     & &9216(8 + 9x^2 - 17x^4)+192\,p^3(1 - 6x -\\
                    & &25x^2 + 6x^3 + 24x^4)- 1152\,p^2(32 -9x + \\ 
                    & &35x^2 - x^3 - 70x^4)+ p^6(37- 102x + 48x^2 +\\ 
                    & & 72x^3+144x^4)+ 24\,p^4(192 - 91x +196x^2 -\\ 
                    & & 24x^3 - 456x^4)\Big).
                \end{eqnarray*}
                Following the same procedure as in the preceding case for the face $y=0$, we conclude that the system of equations $\partial h_6/\partial x = 0$ and $\partial h_6/\partial p = 0$ admit no solution in $(0,2) \times (0,1)$.
            \end{enumerate}
    
        \item \textbf{On the Edges of S.}\\
            Now we determine the maximum values attained by the function $F(p,x,y)$ along the edges of the cuboid $S$.\\
            Setting $x=0$ and $y=0$, the function $F(p,x,y)$ reduces to a single-variable function of $p$, given by
            \begin{equation}
                k_1(p):=F(p,0,0)=\dfrac{1}{663552}\bigl(37p^6+2592p^2(4-p^2)\bigr). \nonumber
            \end{equation}
            Solving $\partial k_1/\partial p=0$, we obtain the critical point $p:=p_0\approx1.44702$. A straightforward computation shows that $k_1(p)$ attains its maximum value at $p_0$, where $k_1(p_0)\approx0.0161025$. Hence,
            \[
                F(p,0,0)\leq0.0161025.
            \]
\noindent
Next, taking $x=0$ and $y=1$, the function $F(p,x,y)$ reduces to
\begin{equation}
k_2(p):=F(p,0,1)
=\dfrac{1}{663552}\bigl(37p^6+48p^3(4-p^2)+4608(4-p^2)^2\bigr). \nonumber
\end{equation}
Solving $\partial k_2/\partial p=0$, we obtain the critical point $p:=p_0=0$. Since $k_2(p)$ is a decreasing function in the admissible domain, its maximum value is attained at $p_0$, where $k_2(p_0)=1/9$. Therefore,
\[
F(p,0,1)\leq\frac{1}{9}.
\]
\noindent
Now, setting $p=0$ and $x=0$, we have
\begin{equation}
F(0,0,y)=\dfrac{y^2}{9}\le\dfrac{1}{9}, \quad y\in[0,1]. \nonumber
\end{equation}
\noindent
Taking $x=1$ and $y=1$, the function $F(p,x,y)$ reduces to
\begin{equation}
k_3(p):=F(p,1,1)
=\dfrac{1}{663552}\bigl(p^2(14976-4392p^2+199p^4)\bigr). \nonumber
\end{equation}
It is evident that $F(p,1,1)=F(p,1,0)$. Solving $\partial k_3/\partial p=0$, we find that $k_3(p)$ attains its maximum value at $p:=p_0\approx1.40293$, where $k_3(p_0)\approx0.0210673$ and consequently,
\[
F(p,1,1)=F(p,1,0)\leq0.0210673.
\]
\noindent
Further, taking $p=0$ and $x=1$, we obtain
\begin{equation}
F(0,1,y)=0. \nonumber
\end{equation}
\noindent
Evaluating $F(p,x,y)$ at $p=2$, we observe that
\begin{equation}
F(2,1,y)=F(2,0,y)=F(2,x,0)=F(2,x,1)
=\dfrac{37}{10368}\approx0.00356867. \nonumber
\end{equation}
\noindent
Next, taking $p=0$ and $y=1$, the function reduces to
\begin{equation}
k_4(x):=F(0,x,1)
\le\dfrac{1}{9}, \quad x\in[0,1]. \nonumber
\end{equation}
\noindent
Finally, taking $p=0$ and $y=0$, we obtain
\begin{equation}
k_5(x):=F(0,x,0)=\dfrac{1}{8}x(1-x^2). \nonumber
\end{equation}
Solving $\partial k_5/\partial x=0$, we find the critical point $x:=x_0\approx0.57735$. A direct computation shows that $k_5(x)$ attains its maximum value at $x_0$, where $k_5(x_0)\approx0.0481125$. Hence,
\[
F(0,x,0)\leq0.0481125.
\]
\end{enumerate}
Hence, in  view of cases $1$, $2$ and $3$ we get
\begin{equation}
    |H_{3,1}(f)|\leq\dfrac{1}{9}. \nonumber
\end{equation}
And the extremal function for this inequality is $f_4(z)\in \mathcal{S}^*_B,$ defined as
\begin{equation}
 f_{4}(z) = z \exp\left(\int_{0}^{z} \dfrac{\log(1+t^3)}{t(1-\log(1+t^3))} \, dt\right) =z + \dfrac{1}{3} z^4 + \dfrac{5}{36} z^7 + \dfrac{11}{324}z^{10}+\cdots . \nonumber
\end{equation}
This completes the proof.
    \end{proof}

\section*{Krushkal Inequality}
Here, we aim to establish the following inequality:
\[
\left|K_{n,m}(f)\right| := \left| a_n^m - a_2^{\,m(n-1)} \right| \leq 2^{m(n-1)} - n^m,
\]
for functions of the form \eqref{eq:1.1}. Particular attention is devoted to the cases $n = 4,\; m = 1$ and $n = 5,\; m = 1$, which are of special significance in this context. First we'll establish the bound for $|a_4-a_2^3|.$
\begin{theorem}
    Let $f \in \mathcal{S}^*_{B}$, then 
        \begin{equation}
            |a_4-a_2^3|\leq \frac{17}{36}. \nonumber
        \end{equation}
    This inequality is sharp.
\end{theorem}
\begin{proof}
    Let $f\in \mathcal{S}^{\ast}_{B}$. Then using \eqref{a2a3a4 in terms of w} and \eqref{a5inw}, we get
    \begin{equation}
        a_4-a_3^3=\dfrac{1}{36}\left(12b_3+30b_1b_2-17b_1^3\right).\label{a4-a2^3}
    \end{equation}
    Now using lemma\ref{lem3} in \eqref{a4-a2^3}, we obtain
    \begin{eqnarray}
    |a_4-a_3^3|&\leq&\dfrac{1}{36}\left|12b_3+30b_1b_2-17b_1^3\right|\nonumber\\
        &\leq&\dfrac{1}{36}\left(12|b_3|+30|b_1||b_2|+17|b_1|^3\right)\nonumber\\
        &\leq&\dfrac{1}{36}\left(12\left(1-|b_1|^2-\dfrac{|b_2|^2}{1+|b_1|}\right)+30|b_1||b_2|+17|b_1|^3\right)\nonumber
    \end{eqnarray}
    Now taking $x:=|b_1|$ and $y:=|b_2|$, we get
    \begin{eqnarray}
        |a_4-a_3^3|&\leq&\dfrac{1}{36}\left(12\left(1-x^2-\dfrac{y^2}{1+x}\right)+30xy+17x^3\right):=H(x,y).\label{Hx,y}
    \end{eqnarray}
    From \eqref{Hx,y}, we get
    \begin{equation}
        |a_4-a_3^3|\leq H(x,y).
    \end{equation}
    Now we need to find the maximum value of $H(x,y)$ in the region $\Omega:=\{(x,y):\;x\ge0,\;y\ge0,\;y\le1-x^2\}$. First we'll check interior of $\Omega$. From $\partial H/\partial y=0$, we obtain
    \[y=\dfrac{5}{4}(1+x):=y^*\]
    Since $y^*\not\le 1-x^2$ for all $x\in[0,1]$, hence there doesn't exists solution.
    Now we'll check for $y=0$ and $x=0$, we get $H(0,0)=1/3$. and For $y=0$ and $x=1$, we get $H(1,0)=17/36$. Hence from \eqref{Hx,y}, we get
    \[
     |a_4-a_3^3|\leq\dfrac{17}{36}.
    \]
    This inequality is sharp and the extremal for this inequality is obtained by the function $f_5(z)\in \mathcal{S}^{\ast}_{B}$, defined as
    \begin{equation}
 f_{5}(z) = z \exp\left(\int_{0}^{z} \dfrac{\log(1+t^3)}{t(1-\log(1+t^3))} \, dt\right) =z + \dfrac{1}{3} z^4 + \dfrac{5}{36} z^7 + \dfrac{11}{324}z^{10}+\cdots . \label{f3}
\end{equation}
This completes the proof.
\end{proof}

    Now we'll establish the bound for $|a_5-a_2^4|.$
    \begin{theorem}
        Let $f \in \mathcal{S}^*_{B}$, then 
        \begin{equation}
            |a_5-a_2^4|\leq \frac{187}{288}. \nonumber
        \end{equation}
    This inequality is sharp.
    \end{theorem}
    \begin{proof}
    Let $f\in \mathcal{S}^{\ast}_{B}$. Then using \eqref{a2a3a4 in terms of w} and \eqref{a5inw}, we get
    \begin{equation}
        a_5-a_2^4=\dfrac{1}{288}\left(72b_4+72b_2^2+168b_1b_3+276b_1^2b_2-187b_1^4\right).\label{a5-a2^4}
    \end{equation}
    Now using lemma\ref{lem3} in \eqref{a5-a2^4}, we obtain
    \begin{eqnarray}
    |a_5-a_2^4|&\leq&\dfrac{1}{288}\left|72b_4+72b_2^2+168b_1b_3+276b_1^2b_2-187b_1^4\right|\nonumber\\
        &\leq&\dfrac{1}{288}\left(72|b_4|+72|b_2|^2+168|b_1||b_3|+276|b_1|^2|b_2|+187|b_1|^4\right)\nonumber\\
        &\leq&\dfrac{1}{288}\Big(72(1-|b_1|^2-|b_2|^2)+72|b_2|^2+168|b_1|\left(1-|b_1|^2-\dfrac{|b_2|^2}{1+|b_1|}\right)+\nonumber \\
        & & 276|b_1|^2|b_2|+187|b_1|^4\Big)
    \end{eqnarray}
    Now taking $x:=|b_1|$ and $y:=|b_2|$, we get
    \begin{equation}
        |a_5-a_2^4|\leq\dfrac{1}{288}\left(72(1-x^2-y^2)+72y^2+168x\left(1-x^2-\dfrac{y^2}{1+x}\right)+276x^2y+187x^4\right) \label{H2x,y}
    \end{equation}
    Now take $H^*(x,y):=\dfrac{1}{288}\left(72(1-x^2-y^2)+72y^2+168x\left(1-x^2-\dfrac{y^2}{1+x}\right)+276x^2y+187x^4\right)$
    From \eqref{H2x,y}, we get
    \begin{equation}
        |a_5-a_2^4|\leq H^*(x,y).
    \end{equation}
    Now we need to find the maximum value of $H^*(x,y)$ in the region $\Omega:=\{(x,y):\;x\ge0,\;y\ge0,\;y\le1-x^2\}$. First we'll check interior of $\Omega$. From $\partial H^*/\partial y=0$, we obtain
    \[y=\dfrac{23}{28}x(1+x):=y^*\]
    Since $y^*\not\le 1-x^2$ for all $x\in[0,1]$, hence there doesn't exists solution.
    Now we'll check for $y=0$ and $x=0$, we get $H^*(0,0)=1/4$. and For $y=0$ and $x=1$, we get $H^*(0,1)=187/288\approx0.6493$. Hence from \eqref{H2x,y}, we get
    \[
     |a_5-a_2^4|\leq\dfrac{187}{288}.
    \]
    This inequality is sharp and the extremal for this inequality is obtained by the function $f_1(z)\in \mathcal{S}^{\ast}_{B}$, defined as
    \begin{equation}
    f_{1}(z) = z \exp\left(\int_{0}^{z} \dfrac{\log(1+t)}{t(1-\log(1+t))} \, dt\right) =z + z^2+\dfrac{3}{4}z^3+\dfrac{19}{36} z^4+\dfrac{101}{288}z^5+ \dfrac{551}{2400} z^6+\cdots . \label{f4}
    \end{equation}
    This completes the proof.
\end{proof}
\subsection*{Toeplitz Determinant}

Now, we investigate the third-order Toeplitz determinant for functions belonging to the class $\mathcal{S}^*_{B}$. Sharp bound for the determinant is obtained, and the corresponding extremal function is identified.

\begin{theorem}
Let $f \in \mathcal{S}^*_{B}$. Then 
\begin{equation}
|T_{3,1}(f)| \leq 1. \nonumber
\end{equation}
The above estimate is sharp.
\end{theorem}

\begin{proof}
Let $f \in \mathcal{S}^*_{B}$. Then, in view of \eqref{toeplitz} for third-order Toeplitz determinant, we have
\begin{equation}
T_{3,1}(f) = 1 - 2a_2^2 - a_3^2 + 2a_2^2 a_3. \nonumber
\end{equation}
Substituting the coefficient expressions given in \eqref{a2} and \eqref{a3} into the above identity, we obtain
\begin{eqnarray}
T_{3,1}(f)
&=& \frac{1}{256}\left(7p_1^4 - 128p_1^2 - 16p_2^2 + 24p_1^2 p_2 + 256 \right) \nonumber\\
&=& \frac{1}{256}\left(256 - 128p_1^2 + 15p_1^4
+ 16(p_1^2 - p_2)\left(p_2 - \tfrac{1}{2}p_1^2\right)\right).
\label{98}
\end{eqnarray}
Taking absolute values on both sides of \eqref{98}, we arrive at
\begin{equation}
|T_{3,1}(f)|\leq\frac{1}{256}\left|256 - 128p_1^2 + 15p_1^4+ 16(p_1^2 - p_2)\left(p_2 - \tfrac{1}{2}p_1^2\right)\right|. \nonumber
\end{equation}
Applying~\cref{lemma2.1}, we conclude that
\begin{equation}
|T_{3,1}(f)| \leq 1. \nonumber
\end{equation}
This establishes the desired upper bound. To show that the result is sharp, we consider the extremal function
\begin{equation}
f_{3}(z)
= z \exp\left(\int_{0}^{z} \frac{\log(1+t)}{t\bigl(1-\log(1+t)\bigr)} \, dt\right)= z + z^2 + \frac{3}{4}z^3 + \frac{19}{36}z^4 + \frac{101}{288}z^5 + \cdots. \nonumber
\end{equation}

A direct computation shows that equality holds for this function, thereby confirming the sharpness of the estimate.
\end{proof}


\subsection*{Hermitian-Toeplitz Determinant}
We now proceed to investigate the third-order Hermitian-Toeplitz determinant for functions belonging to the class $\mathcal{S}^*_{B}$. Sharp upper and lower bounds for the determinant are derived, and the corresponding extremal function is identified.
\begin{theorem}
Let $f \in \mathcal{S}^*_{B}$, then 
\begin{equation}
 -\dfrac{1}{16}\leq H^T_{3,1}(f)\leq 1. \nonumber
\end{equation}
This inequality is sharp.
\end{theorem}
\begin{proof}
Let $f \in \mathcal{S}^*_{B}$. From \eqref{T31f}, the third-order Hermitian-Toeplitz determinant can be expressed as
\begin{equation}
H^T_{3,1}(f)=2\Re(a_2^2\overline{a_3})-2|a_2|^2-|a_3|^2+1. \label{15}
\end{equation}
Using \eqref{a2}, \eqref{a3} and~\cref{lemma2.1}, we obtain
\begin{eqnarray}
2\Re(a_2^2\,\overline{a_3})
&=&2\,\Re \left(\dfrac{p_1^2}{16}\left(\dfrac{1}{6}\overline{p_2}-\dfrac{1}{24}p_1^2\right)\right)
=\dfrac{1}{192}p_1^4+\dfrac{1}{96}p_1^2(4-p_1^2)\Re(\overline{\gamma}), \label{17}\\
2|a_2|^2
&=&\dfrac{1}{2}p_1^2, \label{16}\\
|a_3|^2
&=&\left|\dfrac{1}{6}p_2-\dfrac{1}{24}p_1^2\right|^2
=\left|\dfrac{1}{24}p_1^2+\dfrac{1}{12}(4-p_1^2)\gamma\right|^2 \nonumber\\
&=&\dfrac{1}{576}p_1^4+\dfrac{1}{144}p_1^2(4-p_1^2)\Re (\gamma)
+\dfrac{1}{144}(4-p_1^2)^2|\gamma|^2. \label{18}
\end{eqnarray}
Substituting \eqref{17}, \eqref{16}, and \eqref{18} into \eqref{15}, we arrive at
\begin{eqnarray}
H^T_{3,1}(f)
&=&\dfrac{1}{192}p_1^4+\dfrac{1}{96}p_1^2(4-p_1^2)\Re(\overline{\gamma})-\dfrac{1}{2}p_1^2-\dfrac{1}{576}p_1^4+\dfrac{1}{144}p_1^2(4-p_1^2)\Re (\gamma)+ \nonumber\\
& &\quad \dfrac{1}{144}(4-p_1^2)^2|\gamma|^2+1 \nonumber\\
&=&\dfrac{15}{256}p_1^4-\dfrac{1}{2}p_1^2+\dfrac{1}{64}p_1^2(4-p_1^2)\Re(\gamma)-\dfrac{1}{64}(4-p_1^2)^2|\gamma|^2+1 \nonumber\\
&=&\dfrac{1}{256}\Big(256+15p_1^4-128p_1^2+4p_1^2(4-p_1^2)\Re (\gamma)-4(4-p_1^2)^2|\gamma|^2\Big). \nonumber
\end{eqnarray}
Since $\Re (\gamma)\leq|\gamma|$, it follows that
\begin{equation}
H^T_{3,1}(f)
\leq\dfrac{1}{256}\Big(256+15p_1^4-128p_1^2
+4p_1^2(4-p_1^2)|\gamma|
-4(4-p_1^2)^2|\gamma|^2\Big)
:=\dfrac{1}{256}h(p_1,\gamma), \label{12}
\end{equation}
where $p_1:=p\in[0,2]$ and $|\gamma|:=x\in[0,1]$. Under this notation, \eqref{12} reduces to
\begin{equation}
h(p,x)=256+15p^4-128p^2+4p^2(4-p^2)x-4(4-p^2)^2x^2. \nonumber
\end{equation}
Consequently, we may write
\begin{equation}
\dfrac{1}{256}\min h(p,x)
\leq H^T_{3,1}(f)
\leq \dfrac{1}{256}\max h(p,x). \label{minmax}
\end{equation}
\begin{figure}[H]
    \centering
    \includegraphics[width=0.4\textwidth]{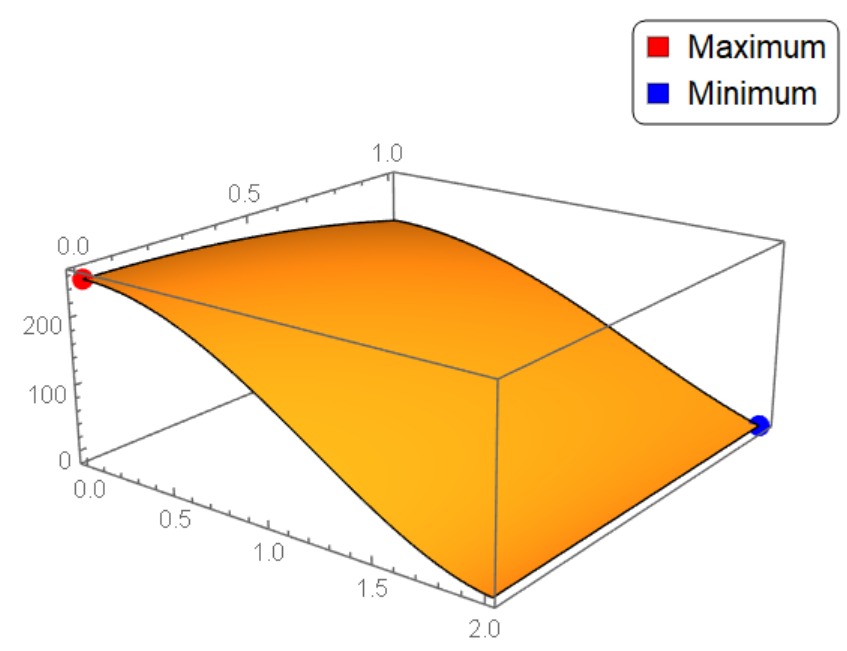} 
    \caption{ $3D$ plot of $h(p,x)$, for $p\in[0,2],\; x \in [0,1]$.}
    \label{fig:phi}
\end{figure}
A straightforward computation shows that maximum value of $h(p,x)$ attains at $(p,x)=(0,0)$ and minimum value attains at $(p,x)=(2,1/2)$ for $p\in[0,2],\;x\in[0,1]$ i.e.
\begin{equation}
\max h(p,x)=256, \quad \min h(p,x)=-16. \label{19}
\end{equation}

Using \eqref{19} in \eqref{minmax}, we obtain
\begin{equation}
H^T_{3,1}(f)\leq\dfrac{1}{256}\max h(p,x)=1, \nonumber
\end{equation}
and
\begin{equation}
H^T_{3,1}(f)\ge\dfrac{1}{256}\min h(p,x) \nonumber
=-\dfrac{1}{16}.
\end{equation}
Hence,
\begin{equation}
-\dfrac{1}{16}\leq H^T_{3,1}(f)\leq 1. \nonumber
\end{equation}
The extremal function corresponding to the upper bound is given by
\begin{equation}
f_{2}(z)
= z \exp\left(\int_{0}^{z} \dfrac{\log(1+t^4)}{t(1-\log(1+t^4))} \, dt\right)
= z + \dfrac{1}{4} z^5 + \dfrac{1}{32} z^9 + \cdots, \nonumber
\end{equation}
while the extremal function for the lower bound is
\begin{equation}
f_{3}(z)
= z \exp\left(\int_{0}^{z} \dfrac{\log(1+t)}{t(1-\log(1+t))} \, dt\right)
= z + z^2 + \dfrac{3}{4} z^3 + \dfrac{19}{36} z^4
+ \dfrac{101}{288} z^5 + \cdots. \nonumber
\end{equation}
This completes the proof.
\end{proof}

\section*{Conclusion}
In the present investigation, we studied several sharp coefficient estimate problems associated with the class $\mathcal{S}^*_{B}$. These results demonstrate a strong connection between coefficient functionals and determinant problems, highlighting the influence of the underlying geometric domain on these quantities. The results established in this paper also have potential implications beyond pure function theory. For instance, convexity and disk-type conditions are relevant in areas such as control theory and motion planning, where they are related to stability and constraint regions. Similarly, Hankel, Toeplitz determinants arise in image processing, signal processing and discrete quantum mechanics, while the Krushkal inequality and Zalcman functional are connected to quasiconformal mapping theory and fluid flow models. These findings open avenues for further application oriented research. Moreover, the techniques developed in this paper are expected to be applicable to other related subclasses defined via different geometric or analytic conditions.


\begin{thebibliography}{99}

\bibitem{Banga_third_hankel} Banga, S., Kumar, S.~S.: The sharp bounds of the second and third Hankel determinants for the class $\mathcal{SL}^*$. Math. Slovaca 70(4), 849-862 (2020)

\bibitem{caratheodory1907} Carath\'eodory, C.: \"{U}ber den Variabilit\"{a}tsbereich der Koeffizienten von Potenzreihen, die gegebene Werte nicht annehmen. Math. Ann. 64(1), 95-115 (1907)

\bibitem{cho_kumar_kumar_ravichandran2019} Cho, N.~E., Kumar, V., Kumar, S.~S., Ravichandran, V.: Radius problems for starlike functions associated with the sine function. Bull. Iranian Math. Soc. 45(1), 213-232 (2019) \href{https://doi.org/10.1007/s41980-018-0127-5}{https://doi.org/10.1007/s41980-018-0127-5}

\bibitem{cho_kumar_kumar2022} Cho, N.~E., Kumar, S., Kumar, V.: Hermitian-Toeplitz and Hankel determinants for certain starlike functions. Asian Eur. J. Math. 15(3), 2250042(2022)

\bibitem{giri2025} Giri, S.: Second-Order Toeplitz Determinant for Starlike Mappings in One and Higher Dimensions. Anal. Math. Phys. 15(4), 96(2025) \href{https://doi.org/10.1007/s13324-025-01098-y}{https://doi.org/10.1007/s13324-025-01098-y}

\bibitem{Goodman1983} Goodman, A.~W.: Univalent Functions. Mariner Publishing Company Inc. Tampa (1983)

\bibitem{Jastrzebski_hermitian} Jastrzebski, P., Kowalczyk, B., Kwon, O.~S., Lecko, A., Sim, Y.J.: Hermitian-Toeplitz determinants of the second and third-order for classes of close-to-star functions. Rev. Real Acad. Cienc. Exactas Físicas Nat. Ser. Matemt\'aicas 114, 1-14 (2020) \href{https://doi.org/10.https://doi.org/10.1007/s13398-020-00895-31007/s13398-020-00895-3}{https://doi.org/10.1007/s13398-020-00895-3}

\bibitem{Kowalczyk_and_lecko_third_hankel} Kowalczyk, B., Lecko, A., Thomas, D.~K.: The sharp bound of the third Hankel determinant for starlike functions. Forum Math. 34(5), 1249-1254 (2022)

\bibitem{kumar_verma2024} Kumar, S.~S., Verma, N.: On a Subclass of Starlike Functions Associated With a Strip Domain. Ukrainian Math. J. 76(12), 1783-1801 (2024) \href{https://doi.org/10.1007/s11253-025-02434-y}{https://doi.org/10.1007/s11253-025-02434-y}

\bibitem{kumar_arya_balloon} Kumar, S.~S., Tripathi, A.: Coefficient problems of Starlike Functions Related to a balloon-shaped Domain. Iran.  J. Sci. (In Press)

\bibitem{kwon_adam_p1p2lemma} Kwon, O.~S., Lecko, A., Sim, Y.~J.: On the fourth coefficient of functions in the Carath\'{e}odory class. Comput. Methods Funct. Theory 18(2), 307-314 (2018) \href{https:/https://doi.org/10.1007/s40315-017-0229-8/doi.org/10.1007/s40315-017-0229-8}{https://doi.org/10.1007/s40315-017-0229-8}

\bibitem{Lecko2020} Lecko, A., Sim, Y.~J., \'Smiarowska, B.: The fourth-order Hermitian Toeplitz determinant for convex functions. Anal. Math. Phys. 10(3), 39(2020) \href{https://doi.org/10.1007/s13324-020-00382-3}{https://doi.org/10.1007/s13324-020-00382-3}

\bibitem{libera_p1p2lemma} Libera, R.~J., Z\l otkiewicz, E.~J.: Early coefficients of the inverse of a regular convex function. Proc. Amer. Math. Soc. 85(2), 225-230 (1982)

\bibitem{MaMinda1994} Ma, W.: A unified treatment of some special classes of univalent functions. In: Proceedings of the Conference on Complex Analysis (Tianjin, 1992), 157-169, Conf. Proc. Lecture Notes Anal., I Int. Press, Cambridge (1992)

\bibitem{Obradovic_hermitian} Obradovi\'c, M., Tuneski, N.: Hermitian-Toeplitz determinants for the class ${\mathcal {S}} $ of univalent functions. Armen. J. Math. 13(4), 1-10 (2021) \href{https://doi.org/10.52737/18291163-2021.13.4-1-10}{https://doi.org/10.52737/18291163-2021.13.4-1-10}

\bibitem{Pommerenke1967} Pommerenke, C.: On the Hankel determinants of univalent functions. Mathematika 14(1), 108-112 (1967) 

\bibitem{Pommerenke1975} Pommerenke, C.: Univalent Functions. Vandenhoeck and Ruprecht, G\"ottingen (1975)

\bibitem{Rath_and_lecko_third_hankel} Rath, B., Kumar, K.~S., Krishna, D.~V., Lecko, A.: The sharp bound of the third Hankel determinant for starlike functions of order $1/2$. Complex Anal. Oper. Theory 16(5), 65(2022) \href{https://doi.org/10.1007/s11785-022-01241-8}{https://doi.org/10.1007/s11785-022-01241-8}

\bibitem{riaz_lecko_2023} Riaz, A., Lecko, A., Raza, M.: Starlikeness Associated with Cosine Hyperbolic Function. Iran. J. Sci. 47(5), 1723-1738 (2023) \href{https://doi.org/10.1007/s40995-023-01539-y}{https://doi.org/10.1007/s40995-023-01539-y}

\bibitem{Sokol1996} Sok\'o\l, J., Stankiewicz, J.: Radius of convexity of some subclasses of strongly starlike functions. Zesz. Nauk. Politech. Rzesz. Mat. 19, 101-105 (1996)

\bibitem{sarkar_hermitian_toeplitz} Sarkar, N.: Bounds on Hermitian-Toeplitz determinant for starlike, convex and bounded turning functions associated with the exponential function. Mat. Stud. 64(2), 115-123 (2025)

\bibitem{zaprawa2021thirdhankel} Zaprawa, P., Obradovi\'c, M., Tuneski, N.: Third Hankel determinant for univalent starlike functions. Revista de la Real Academia de Ciencias Exactas, F\'isicas y Naturales. Serie A. Matem\'aticas 115(2), 49(2021) \href{https://doi.org/10.1007/s13398-020-00977-2}{https://doi.org/10.1007/s13398-020-00977-2}

\bibitem{zaprawa2021} Zaprawa, P.: Initial logarithmic coefficients for functions starlike with respect to symmetric points. Bol. Soc. Mat. Mex. 27(62), (2021)


\end{thebibliography}
\end{document}